\documentclass{elsarticle}

\usepackage{graphicx}
\usepackage{amssymb,amsmath}
\usepackage{amsthm}
\usepackage{color}
\usepackage{bm}
\usepackage{graphicx}
\usepackage{algorithm}
\usepackage{enumerate}

\newtheorem{theorem}{Theorem}[section]

\newtheorem{lemma}[theorem]{Lemma}
\newtheorem{remark}[theorem]{Remark}

\newcommand{\ds}{\displaystyle}

\begin{document}

\begin{frontmatter}

\title{{\color{blue}An equivalent formulation of Sonine condition}}

\author[1]{Xiangcheng Zheng}
\ead{xzheng@sdu.edu.cn}

\address[1]{School of Mathematics, Shandong University, Jinan 250100, China}

\journal{Applied Mathematics Letters}

\begin{abstract}
Sonine kernel is characterized by the Sonine condition (denoted by SC) and is an important class of kernels in nonlocal differential equations and integral equations. {\color{blue}This work proposes a SC with a more general form (denoted by gSC), which is more convenient than SC to accommodate complex kernels and equations.} A typical kernel is given, and the first-kind Volterra integral equation under gSC is accordingly transformed and then analyzed. {\color{blue}Based on these results, it is finally proved that the gSC is indeed equivalent to the original SC, which indicates that the Sonine kernel may be essentially characterized by the behavior of its convolution with the associated kernel at the starting point}.
\end{abstract}

\end{frontmatter}

\section{Introduction}
The Sonine kernel, which is firstly addressed in \cite{Son}, is an important class of kernels in nonlocal differential equations \cite{Luc,LucFCAA} and integral equations \cite{Car,Sam}. For $b>0$, a kernel $k(t)\in L^1(0,b)$ is called a Sonine kernel if there exists a kernel $K(t)\in L^1(0,b)$ such that the relation
\begin{equation}\label{mh1}
\int_0^tK(t-s)k(s)ds=1
\end{equation}
is valid for almost all $t\in (0,b)$. In this case, the relation (\ref{mh1}) is called the Sonine condition (denoted by SC) and the $K(t)$ is called the associate kernel of $k(t)$.

 Apart from extensive applications, a mathematical advantage of the Sonine kernel lies in that the solutions to the corresponding integral or nonlocal differential equations could be explicitly expressed via the associate kernel. {\color{blue}For instance, the study of the integral equation of the first kind \cite{Gor}
\begin{equation}\label{mh01}
\int_0^tk(t-s)u(s)ds=f(t)
\end{equation}
has a long history, and a main difficulty is to express its solutions for the sake of analysis.}
If $k(t)$ in (\ref{mh01}) is a Sonine kernel, then the solution could be expressed as
\begin{equation}\label{mh02}
u(t)=\frac{d}{dt}\int_0^tK(t-s)f(s)ds. 
\end{equation}
{\color{blue}Based on this solution expression, various investigations have been performed for the integral equation (\ref{mh01}), especially for the weakly singular case $k(t)=t^{-\alpha}$ with $0<\alpha<1$ that leads to the Abel integral equation, see e.g. \cite{Gor} for the comprehensive analysis and applications of Abel integral equations.

However, if the kernel $k(t)$ in the integral equation (\ref{mh01}) has a complex form, then it could be difficult to determine an associated kernel $K(t)$ such that the SC (\ref{mh1}) is satisfied. A typical example is the variable-exponent Abel kernel $k(t)=t^{-\alpha(t)}$ for some smooth exponent function $0<\alpha(t)<1$, which attracts increasing attentions in nonlocal models, cf. a comprehensive review of the variable-exponent models in \cite{SunChaZha}. Due to the variability of $\alpha(t)$, it is not convenient to find an associated kernel of $k(t)=t^{-\alpha(t)}$, that is, one could not determine whether it is a Sonine kernel or not based on the SC (\ref{mh1}). 
}

 {\color{blue}Motivated by this issue, we introduce a Sonine condition with a more general form (denoted by gSC), which is more convenient than the SC (\ref{mh1}) to accommodate complex kernels and equations. Specifically, it is easier to check whether a complex kernel satisfies the gSC or not.}  Although the gSC could introduce difficulties in expressing the solutions explicitly such as (\ref{mh02}), it still helps to convert the original model to a more feasible form to carry out analysis.  We propose a typical example via the variable-exponent Abel kernel $t^{-\alpha(t)}$ for demonstration, and then apply the gSC to analyze the first-kind Volterra integral equation. Based on these results, we finally prove that the gSC is indeed equivalent to the original SC (\ref{mh1}), which indicates that the Sonine kernel may be essentially characterized by the behavior of its convolution with the associated kernel at the starting point.


\section{A more general form}
For $k(t)\in L^1(0,b)$, the gSC reads: there exists a $K(t)\in L^1(0,b)$ such that
\begin{equation}\label{mh2}
\begin{array}{c}
\ds \int_0^tK(t-s)k(s)ds=g(t)\text{ for some differentiable function }g(t)\\
\ds\text{ with } g'(t)\in L^1(0,b)\text{ and }g(0)=1.
\end{array}
\end{equation}
Under this relation, $k(t)$ is called the Sonine kernel under the gSC with the associated kernel $K(t)$. Compared with the original SC, the gSC requires $g(t)=1$ only at $t=0$ and gives more flexibility for $g(t)$ on $t\in (0,b]$, which in turn admits more degree of freedom in finding the associated kernel. When $g(t)\equiv 1$, the gSC degenerates to the classical SC (\ref{mh1}). \\
\textbf{Example: The variable-exponent Abel kernel}\\
A typical example is the following variable-exponent Abel kernel 
\begin{equation}\label{kt}
k(t)=t^{-\alpha(t)} \text{ for some }0<\alpha(t)<1\text{ on }[0,b],
\end{equation}
where the variable exponent $\alpha(t)$ describes, e.g. the variation of the memory and hereditary properties \cite{Boc,Lag,LiaSty,SunChaZha,ZenZha}. {\color{blue}Due to the variability of the exponent, it is difficult to find the associated kernel $K(t)$ of $k(t)=t^{-\alpha(t)}$ satisfying the original SC (\ref{mh1}), but it is possible to determine a $K(t)$ that fulfills the gSC (\ref{mh2}) as follows.}
To simplify the derivations, we assume that 
$$\alpha(t)\text{ is differentiable with }|\alpha'(t)|\leq L. $$

\begin{theorem}
For $k(t)$ defined in (\ref{kt}), there exists an associated kernel
\begin{equation*}
K(t)=\frac{t^{\alpha(0)-1}}{\kappa} \text{ with }\kappa=\Gamma(\alpha(0))\Gamma(1-\alpha(0))
\end{equation*}
satisfying the gSC (\ref{mh2}). Here $\Gamma(\cdot)$ denotes the Gamma function.
\end{theorem}
\begin{proof} 
It is clear that $K(t)\in L^1(0,b)$. To verify other requirements, direct calculation yields
\begin{align}
\int_0^tK(t-s)k(s)ds&=\frac{1}{\kappa}\int_0^t(t-s)^{\alpha(0)-1}s^{-\alpha(s)}ds\nonumber\\
&\overset{z=s/t}{=}\frac{1}{\kappa}\int_0^1(t-tz)^{\alpha(0)-1}(tz)^{-\alpha(tz)}tdz\label{mh7}\\
&=\frac{1}{\kappa}\int_0^1(tz)^{\alpha(0)-\alpha(tz)}(1-z)^{\alpha(0)-1}z^{-\alpha(0)}dz=:g(t).\nonumber
 \end{align}
Then we intend to prove $g'(t)\in L^1(0,b)$. By the assumptions on $\alpha(t)$ and $x|\ln x|\leq C$ for $x\in [0,b]$ for some constant $C$, we bound
$$(tz)^{\alpha(0)-\alpha(tz)}=e^{(\alpha(0)-\alpha(tz))\ln(tz)}\leq e^{Ltz|\ln(tz)|}\leq e^{LC}. $$ 
We then apply this and the property of $\alpha(t)$ to bound
 \begin{align*}
 \Big|\frac{d}{dt}(tz)^{\alpha(0)-\alpha(tz)}\Big|&=\Big|(tz)^{\alpha(0)-\alpha(tz)}z\Big(-\alpha'(tz)\ln(tz)+\frac{\alpha(0)-\alpha(tz)}{tz}\Big) \Big|\\
 &\leq e^{LC}\big(L|\ln(tz)|+L\big)=e^{LC}L\big(|\ln(tz)|+1\big).
 \end{align*}
 For any fixed $t>0$, the derivative of the integand of $g(t)$ with respect to $t$ could thus be bounded as
 $$B(t,z):=e^{LC}L\big(|\ln(tz)|+1\big)(1-z)^{\alpha(0)-1}z^{-\alpha(0)}. $$
 Since
 \begin{align*}
 \int_0^1B(t,z)dz&\leq e^{LC}L \int_0^1 \big(|\ln(tz)|+1\big)(1-z)^{\alpha(0)-1}z^{-\alpha(0)}dz\\
 &= e^{LC}Lt^{-\epsilon} \int_0^1 (tz)^{\epsilon}\big(|\ln(tz)|+1\big)(1-z)^{\alpha(0)-1}z^{-\alpha(0)-\epsilon}dz\\
 &\leq Qe^{LC}Lt^{-\epsilon}\int_0^1 (1-z)^{\alpha(0)-1}z^{-\alpha(0)-\epsilon}dz\\
 &=Qe^{LC}Lt^{-\epsilon}\frac{\Gamma(\alpha(0))\Gamma(1-\alpha(0)-\epsilon)}{\Gamma(1-\epsilon)}
 \end{align*}
 for $0<\epsilon<1-\alpha(0)$ for any fixed $t>0$ where $Q$ represents the bound of $x^\epsilon(|\ln x|+1)$ for $x\in [0,b]$, we differentiate $g$ to obtain
\begin{align*}
|g'(t)|&=\Big|\frac{1}{\kappa}\int_0^1\frac{d}{dt}(tz)^{\alpha(0)-\alpha(tz)}(1-z)^{\alpha(0)-1}z^{-\alpha(0)}dz\Big|\\
&\leq \frac{1}{\kappa} \int_0^1B(t,z)dz\leq \frac{1}{\kappa}Qe^{LC}Lt^{-\epsilon}\frac{\Gamma(\alpha(0))\Gamma(1-\alpha(0)-\epsilon)}{\Gamma(1-\epsilon)},
\end{align*}
which implies $g'(t)\in L^1(0,b)$. 
%

Furthermore, we apply 
$$\lim_{x\rightarrow 0^+}x^{\alpha(0)-\alpha(x)}=\lim_{x\rightarrow 0^+}e^{(\alpha(0)-\alpha(x))\ln x}=e^0=1 $$
and take $t=0$ in the right-hand side of (\ref{mh7}) to obtain
$$g(0)=\frac{1}{\kappa}\int_0^1(1-z)^{\alpha(0)-1}z^{-\alpha(0)}dz=1.$$
Thus we complete the proof.
\end{proof}

\section{Equivalence via integral equation}
Under the setting of the gSC, we reformulate the integral equation (\ref{mh01}) into a more feasible form for the sake of analysis. Throughout this section, we assume that $f\in C^1[0,b]$ for simplicity, though more careful treatments could be applied to relax this requirement.
\begin{theorem}\label{thm}
The first-kind Volterra integral equation (\ref{mh01}) with a kernel $k$ satisfying the gSC could be formally reformulated to a second-kind 
Volterra integral equation
\begin{equation}\label{mh10}
u(t)+\int_0^t g'(t-y)u(y)dy=\frac{d}{dt}\int_0^tK(t-s)f(s)ds.
\end{equation}
\end{theorem}
\begin{proof}
We replace $t$ and $s$ in (\ref{mh01}) by $s$ and $y$, respectively, multiple $K(t-s)$ on both sides of the resulting equation, and then integrate from $0$ to $t$ to obtain
\begin{equation}\label{mh3}
\int_0^tK(t-s)\int_0^sk(s-y)u(y)dyds=\int_0^tK(t-s)f(s)ds. 
\end{equation}
We interchange the double integral to get
\begin{equation*}
\int_0^t\int_y^t K(t-s)k(s-y)dsu(y)dy=\int_0^tK(t-s)f(s)ds. 
\end{equation*}
We apply the variable substitution $z=s-y$ for the inner integral to get
\begin{equation*}
\int_0^t\int_0^{t-y} K(t-y-z)k(z)dzu(y)dy=\int_0^tK(t-s)f(s)ds. 
\end{equation*}
Applying the gSC (\ref{mh2}) leads to
\begin{equation}\label{mh6}
\int_0^t g(t-y)u(y)dy=\int_0^tK(t-s)f(s)ds. 
\end{equation}
We finally differentiate this equation on both sides and use $g(0)=1$ to obtain (\ref{mh10}) and thus complete the proof.
\end{proof}
Based on this theorem, one could analyze the transformed equation (\ref{mh10}) instead of the original equation (\ref{mh01}), which significantly reduces the difficulties. However, in order to bring the analysis results on the transformed equation (\ref{mh10}) back to the original equation (\ref{mh01}), the following auxiliary result is needed.
\begin{lemma}\label{thm2}
An $L^1$ solution of the transformed equation (\ref{mh10}) is also a solution to the original equation (\ref{mh01}).
\end{lemma}
\begin{proof}
 Let $u(t)\in L^1(0,b)$ be a solution to the transformed equation. Then we rewrite (\ref{mh10}) as
$$\frac{d}{dt}\Big(\int_0^t \big(g(t-y)u(y)-K(t-y)f(y)\big)dy\Big)=0, $$
which implies
\begin{equation}\label{mh11}
\int_0^t \big(g(t-y)u(y)-K(t-y)f(y)\big)dy=c_0 
\end{equation}
for some constant $c_0$.
As $u(t),\,K(t)\in L^1(0,b)$, $f(t)$ is bounded by assumption and $g(t)$ is bounded as follows
$$|g(t)|=\Big|\int_0^tg'(s)ds+g(0)\Big|\leq \|g'\|_{L^1(0,b)}+1, $$
we pass the limit $t\rightarrow 0^+$ in (\ref{mh11}) to get $c_0=0$, which, together with the same derivations as (\ref{mh3})--(\ref{mh6}), leads to
\begin{equation}\label{mh12}
\int_0^tK(t-s)\Big(\int_0^sk(s-y)u(y)dy-f(s)\Big)ds=0. 
\end{equation}
Let 
$$p(t):=\int_0^sk(s-y)u(y)dy-f(s) $$
such that (\ref{mh12}) could be written as
$$\int_0^tK(t-s)p(s)ds=0,  $$
which is indeed a first-kind Volterra integral equation. In particular, as $K(t)$ is also a Sonine kernel under the gSC with the associated kernel $k(t)$, we apply Theorem \ref{thm} on the above equation to get
$$
p(t)+\int_0^t g'(t-y)p(y)dy=0,
$$
which implies 
$$
|p(t)|\leq \int_0^t |g'(t-y)||p(y)|dy.
$$
Then we apply the Gronwall inequality, see e.g. \cite[Lemma 2.7]{Lin} to obtain $p(t)\equiv 0$, i.e. $u$ solves the original equation (\ref{mh01}).
\end{proof}
With the above analysis, one could analyze the integral equation (\ref{mh01}) under the gSC. Specifically, by classical results on the second-kind Volterra integral equations, see e.g. \cite[Theorem 2.3.5]{Gri}, one could obtain the existence and uniqueness of the solutions $u(t)\in L^1(0,b)$ for the transformed equation (\ref{mh10}) based on the properties of $g$ in the gSC and appropriate conditions on $f(t)$, e.g. $f(t)\in C^1[0,b]$ for simplicity. Then we apply Lemma {\color{blue}\ref{thm2}} to conclude that the integral equation (\ref{mh01}) admits an $L^1$ solution. The uniqueness of the solutions to the integral equation (\ref{mh01}) follows from that of the transformed equation (\ref{mh10}). We summarize this well-posedness result in the following theorem.
\begin{theorem}\label{thm:yqbz}
The integral equation (\ref{mh01}) with a kernel satisfying the gSC and $f\in C^1[0,b]$ admits a unique $L^1$ solution.
\end{theorem}
\begin{remark}
Again, the regularity on $f$ in this theorem is strong for the simplicity of the analysis, which could be relaxed by more careful estimates. 
\end{remark}
{\color{blue}Finally, we apply the well-posedness of the integral equation (\ref{mh01}) to show the equivalence of the SC (\ref{mh1}) and the gSC (\ref{mh2}). It is clear that the original SC (\ref{mh1}) implies the gSC (\ref{mh2}). To derive (\ref{mh1}) from (\ref{mh2}), we suppose the gSC (\ref{mh2}) holds for some kernel $k^*(t)\in L^1(0,b)$. By the well-posedness of the integral equation (\ref{mh01}) with the kernel $k(t)=k^*(t)$ and $f(t)\equiv 1$, as stated in Theorem \ref{thm:yqbz}, there exists a unique solution $u(t)\in L^1(0,b)$ to (\ref{mh01}), which means that the kernel $k^*(t)$ is exactly the Sonine kernel with the associated kernel $u(t)$, i.e. the kernel $k^*(t)$ satisfies the original SC (\ref{mh1}). We summarize this result in the following theorem.
\begin{theorem}
A kernel $k(t)\in L^1(0,b)$ satisfies the SC (\ref{mh1}) if and only if it satisfies the gSC (\ref{mh2}).
\end{theorem}
From this equivalence theorem, we may draw a conclusion intuitively that the Sonine kernel may be essentially characterized by the behavior of its convolution with the associated kernel at the starting point, which motivates us to focus the attention around the starting point in the future research.
}

\section*{Acknowledgement}
This work was partially supported by the Taishan Scholars Program of Shandong Province (No. tsqn202306083), the National Natural Science Foundation of China (No. 12301555), and the National Key R\&D Program of China (No. 2023YFA1008903).

\section*{Disclosure statement}
The author declares that they have no conflict of interest.






\begin{thebibliography}{99}

	
	
	\bibitem{Boc} K. Bockstal, M. Zaky, A. Hendy,
On the existence and uniqueness of solutions to a nonlinear variable order time-fractional reaction-diffusion equation with delay. {\it
Commun. Nonlinear Sci. Numer. Simul.} 115 (2022),
106755.

\bibitem{Car} R. Cardoso and S. Samko, Sonine integral equations of the first kind of in $L_p(0, b)$. {\it Fract. Calc. Appl. Anal.} 6 (2003), 235--258.


\bibitem{Gor}  R. Gorenflo and S. Vessella, {\it Abel integral equations. Analysis and applications.} Springer-Verlag, Berlin, 1991.

\bibitem{Gri} G. Gripenberg, S. Londen and O. Staffans, {\it Volterra Integral and Functional Equations}. Encyclopedia Math. Appl. 34, Cambridge University, Cambridge, 1990.

\bibitem{Lag} A. Laghrib and L. Afraites, 
Image denoising based on a variable spatially exponent PDE. {\it Appl. Comput. Harmon. Anal.}68 (2024), 101608.

\bibitem{LiaSty} H. Liang and M. Stynes, A general collocation analysis for weakly singular Volterra integral equations with variable exponent. {\it IMA J. Numer. Anal.} (2023) https://doi.org/10.1093/imanum/drad072

\bibitem{Lin} Y. Lin, Semi-discrete finite element approximations for linear parabolic integro-differential equations with integrable kernels. {\it J. Integral Equ. Appl.} 10 (1998), 51--83.

\bibitem{Luc} Y. Luchko, General fractional integrals and derivatives with the Sonine kernels. {\it Mathematics} 9 (2021), 594. 

\bibitem{LucFCAA}  Y. Luchko, Operational calculus for the general fractional derivatives with the Sonine kernels. {\it Fract. Calc. Appl. Anal.} 24 (2021), 338--375.

\bibitem{Sam} S. Samko and R. Cardoso, Integral equations of the first kind of Sonine type.
{\it Intern. J. Math. Sci.} 57 (2003), 3609--3632.



\bibitem{Son} N. Sonine, Sur la g\'en\'eralisation d'une formule d'Abel. {\it Acta Math.} 4 (1884), 171--176.

\bibitem{SunChaZha} H.~Sun, A.~Chang, Y.~Zhang, W.~Chen, A review on variable-order fractional differential equations: mathematical foundations, physical models, numerical methods and applications. {\em Fract. Calc. Appl. Anal.} 22 (2019), 27--59.

\bibitem{ZenZha}  F. Zeng, Z. Zhang, G. Karniadakis, A generalized spectral collocation method with tunable accuracy for variable-order fractional differential equations. {\it SIAM J Sci. Comput.} 37 (2015), A2710--A2732.

 









     

















\end{thebibliography}
\end{document}